\documentclass[12pt]{amsart}
\usepackage[utf8]{inputenc}
\usepackage{color}
\usepackage{times}
\usepackage[hidelinks]{hyperref}
\usepackage{enumerate,latexsym}
\usepackage{latexsym}
\usepackage{subcaption}

\usepackage{amsmath,amsthm,amsfonts,amssymb}
\usepackage{graphicx}
\usepackage{dsfont}

\usepackage{tikz}
\usetikzlibrary{arrows.meta}
\usetikzlibrary{knots}
\usetikzlibrary{hobby}
\usetikzlibrary{arrows,decorations,decorations.markings}
\usetikzlibrary{external}
\usetikzlibrary{fadings}

\makeatletter
\def\namedlabel#1#2{\begingroup
 #2%
 \def\@currentlabel{#2}%
 \phantomsection\label{#1}\endgroup
}
\makeatother

\usepackage[lite]{amsrefs}

\renewcommand{\PrintDOI}[1]{\href{http://dx.doi.org/\detokenize{#1}}{doi: \detokenize{#1}}%
	\IfEmptyBibField{pages}{, (to appear in print)}{}}

\theoremstyle{plain}
\newtheorem*{theorem*}{Theorem}
\newtheorem*{thmex*}{Theorem~\ref{example}}
\newtheorem*{thmasymp*}{Theorem~\ref{thmAsymp}}
\newtheorem{theorem}{Theorem}[section]

\newtheorem{corollary}[theorem]{Corollary}

\newtheorem{lemma}[theorem]{Lemma}

\newtheorem{proposition}[theorem]{Proposition}
\newtheorem{remark}[theorem]{Remark}

\newtheorem{example}[theorem]{Example}
\newtheorem{question}[theorem]{Question}

\theoremstyle{definition}
\newtheorem{definition}[theorem]{Definition}

\newcommand{\ben}{\begin{enumerate}}
\newcommand{\een}{\end{enumerate}}

\newcommand{\ed}{\end{document}}

\definecolor{rrr}{rgb}{.9,0,.1}

\definecolor{rr}{rgb}{.8,0,.3}

\usepackage{pdfsync}
\graphicspath{ {./images/} }

\usepackage[all]{xy}

\title{ knot groups, quandle extensions and orderability}

\author[Idrissa Ba]{Idrissa Ba}
\email{ba162006@yahoo.fr}

\author[Mohamed Elhamdadi]{Mohamed Elhamdadi} 
\address{Department of Mathematics, 
University of South Florida, Tampa, FL 33620, U.S.A.} 
\email{emohamed@math.usf.edu} 
\urladdr{ http://shell.cas.usf.edu/~emohamed/}

\begin{document}


\begin{abstract}
This paper gives a new way of characterizing L-space $3$-manifolds by using orderability of quandles. Hence, this answers a question of Adam Clay et al. [Question 1.1 of Canad. Math. Bull. 59 (2016), no. 3, 472-482].  We also investigate both the orderability and circular orderability of dynamical extensions of orderable quandles. We give conditions under which the conjugation quandle on a group, as an extension of the conjugation of a bi-orderable group by the conjugation of a right orderable group, is right orderable. We also study the right circular orderability of link quandles.  We prove that the $n$-quandle $Q_n(L)$ of the link quandle of $L$ is not right circularly orderable and hence it is not right orderable.  But on the other hand, we show that there are infinitely many links for which the $p$-enveloping group of the link quandle is right circularly orderable for any prime integer $p$.
\end{abstract}

\maketitle 
\section{Introduction}
Quandles are, in general, non-associative algebraic structures whose axioms are modeled on the three Reidemeister moves in knot theory. The main examples of quandles come from modules over the Laurent
polynomial ring $\mathbb{Z}[t^{\pm 1}]$ and from conjugacy classes in groups.  Quandles were introduced in the 1980s independently by Joyce in \cite{Joyce} and Matveev in \cite{Matveev}.  However, the notion of a quandle can be traced back to the 1940's in the work of Mituhisa Takasaki \cite{Takasaki}.   Joyce and Matveev associated to each oriented knot $K$ a certain quandle $Q(K)$ called the fundamental quandle of the knot $K$ and showed that 
it is a complete invariant up to reversed mirror.  In other words, two knots $K$ and $K'$ are equivalent (up to reverse and mirror image) if and only if the fundamental quandles $Q(K)$ and  $Q(K')$ are isomorphic as quandles.  Since their introduction, quandles and their kin (biquandles, racks and biracks) have been used extensively to obtain invariants of knots in the $3$-space and knotted surfaces in $4$-space (see for example \cite{EN,CEGS, CKS}).  There have been also interest in quandles and their relations to other areas of mathematics such as ring theory \cite{BPS, EFT, ENS, ENSS}, quasigroups and Moufang loops \cite{Elhamdadi}, representation theory \cite{EM, ESZ}, Lie algebras \cite{CCES1, CCES2}, Yang-Baxter equation \cite{CES}, Hopf algebras \cite{AG, CCES2} and  Frobenius algebras \cite{CCEKS}. 

 In the past ten years, the concept of orderability has found intriguing applications in topology. One notable example is the L-space conjecture, as discussed in \cite{BGW} and \cite{Ju}. This conjecture states that a closed, connected, irreducible and orientable $3$-manifold $M$ admits a co-oriented taut foliation if and only if its fundamental group is left-orderable which is also equivalent to $M$ is not an L-space (see \cite{OS} for a definition of an L-space $3$-manifold).  Extensive research has been conducted on this conjecture, yielding strong evidence in its favor.  In \cite{CR}, the authors showed that if the knot group of a nontrivial knot $K\subset \mathbb{S}^3$, is bi-orderable, then performing surgery on $K$ cannot result an L-space. This finding establishes a connection between Heegaard Floer homology and the orderability of the fundamental group of $3$-manifolds. In this paper, we also investigate this relationship by utilizing the fact that the conjugation quandle of a bi-circularly orderable group is right-orderable as a quandle, a result from our previous work (Proposition 3.3 in \cite{BE}).  These recent developments not only contribute to the understanding of orderability in the context of topology but also provide insights into the intricate connections between various mathematical structures and theories. The evidence accumulated through these studies further supports the L-space conjecture and encourages continued exploration of the relationship between orderability and other fundamental aspects of 3-manifold theory \cite{BaC1, BaC2, BRW, Cal}.

The following is the organization of the article.  In section~\ref{Review}, we recall the basics of quandles, give examples and review orderability of groups.  Section~\ref{Circular} deals with orderability of quandles. In Section~\ref{Sec4}, we prove a certain recursive property about circular orderings on quandles, deduce from it some results such as non-circular orderability of $n$-quandles and also show by an example that a homomorphic image of a right circularly orderable quandle may not be  right circularly orderable.  Section~\ref{Sec5} investigates both the orderability and circular orderability of dynamical extensions of orderable quandles.  In Section~\ref{Sec6}, we give conditions under which the conjugation quandle on a group, as an extension of the conjugation of a bi-orderable group by the conjugation of a right orderable group, is right orderable. Section~\ref{Sec7} deals with right circular orderability of link quandles.  We prove that the $n$-quandle $Q_n(L)$ of the link quandle of $L$ is not right circularly orderable and hence it is not right orderable.  But on the other hand, we show that there are infinitely many links $L$ for which the $2$-enveloping group of $Q(L)$ is right circularly orderable.  Furthermore, we show that for sufficiently large prime numbers $n$ and for $(p,q)$ two bridge knots $L$, the $n$-enveloping group Env$_n(Q(L))$ of the quandle $Q(L)$ is right circularly orderable for $p$ congruent to $3$ modulo $4$.  In Section~\ref{Sec8} we investigate the question of which links in the $3$-sphere $\mathbb{S}^3$ have bi-orderable link groups.  
We answer Question 1.1 of \cite{CDN},  which states "can non-bi-orderability of a knot group be determined by examining knot invariants other than the Alexander polynomial".

\section{Review of quandles and orderability}\label{Review}

A {\it quandle} is a non-empty set $Q$ with a binary operation $*$ which satisfies the following conditions:
\begin{enumerate}
	\item For any $t\in Q$, $t*t=t$;
	\item For any $t_1, t_2\in Q$, there exists a unique $t_3\in Q$ such that $t_3*t_2=t_1$; and
	\item For any $t_1, t_2, t_3\in Q$, 
	$(t_1*t_2)*t_3=(t_1*t_3)*(t_2*t_3)$.
\end{enumerate}
For any $s\in Q$, we define the right multiplication map $R_s: Q \rightarrow Q$ given by $t \mapsto t * s$. Then axioms (2) and (3) state that the map $R_s$ is an automorphism of the quandle $Q$ and axiom (1) states that $s$ is a fixed point of $R_s$ for any $s\in Q$.

An element $e$ of a quandle $Q$ is called {\it stabilizer element} if $s*e=s$ for any $s\in Q$. 

A quandle $Q$ is called {\it trivial} if all its elements are stabilizer elements.
A trivial quandle can have an arbitrary number of element in contrast to a trivial group.
 
We call $id_Q$ the identity automorphism of $Q$. A quandle is called {\it involutory} if $R_s\circ R_s=id_Q$ for any $s\in Q$.

By the conditions $(2)$ and $(3)$ of the definition of quandle, there exists a dual binary operation $*^{-1}$ defined by $s*^{-1}r=t$ if $s=t*r$ for any $r, s$ and $t\in Q$. Hence, we have also that $(s*r)*^{-1}r=s$ and $(s*^{-1}r)*r=s$ for any $r, s\in Q$, this is called right cancellation. Therefore, by the condition $(1)$ of the definition of quandle and right cancellation we have that $s*^{-1}s=s$ for any $s\in Q$.

A quandle is called {\it latin} if for any $s\in Q$ the map $L_s:Q\longrightarrow Q$ defined by $L_s(t)=s*t$ is a bijection for any $t\in Q$. A quandle is called {\it semi-latin} if the map $L_s$ is injective for any $s\in Q$.

A subquandle of a quandle, $Q$, is a subset $X\subset Q$ which is closed under the quandle operations.

The following are some examples of quandles. 
\begin{itemize}
	\item Let $G$ be a group.  Then the binary operation $*$ defined by $g*h=hgh^{-1}$ gives a quandle structure on $G$, denoted $\mathrm{Conj}(G)$, and called \emph{conjugation} quandle of $G$. 
	\item The set $Q=G$, where $G$ is a group, with the binary operation $*$ defined by $g*h=hg^{-1}h$ is a quandle, denoted by Core$(G)$, and called \emph{core} quandle of $G$.  If furthermore $G$ is an abelian group then this quandle is called \emph{Takasaki } quandle \cite{Takasaki}. 
	\item If $\phi$ is an automorphism of a group $G$, then the set $Q=G$ with the binary operation $*$ defined by $g*h=\phi(gh^{-1})h$ is a quandle, denoted by Aut$(G, \phi)$, and called \emph{generalized Alexander} quandle of $G$ with respect to $\phi$.
\end{itemize}

Recall that a group $G$ is called \emph{left-orderable} (respectively right-orderable) if there exists a strict total ordering $<$ on $G$ such that $g<h$ implies $fg<fh$ (respectively $gf<hf$) for all $f, g, h \in G$, the relation $<$ is called a \emph{left-ordering} of $G$.  A left-orderable (respectively right-orderable) quandle was defined in the same fashion in \cite{BPS}. A quandle $(Q,*)$ is called left-orderable (respectively right-orderable) if there exists a strict total ordering $\prec$ on $Q$ such that $s\prec t$ implies $r*s\prec r*t$ (respectively $s*r\prec t*r$) for all $r, s, t \in Q$, the relation $\prec$ is called a \emph{left-ordering} (respectively right-ordering) on $Q$.
The notion of a homomorphism between two ordered quandles is natural.  Let $(X,*, \prec_X )$ and $(Y,\diamond, \prec_Y)$ be two right ordered quandles.  A map $f:X \rightarrow Y$ is called an order preserving homomorphism of right ordered quandles if $f$ is a quandle homomorphism (that is,$f(x*y)=f(x) \diamond f(y), \forall x,y \in X$) and $ x \prec_X y$ implies $f(x) \prec_Y f(y)$.

\section{Circular orderability of quandles}\label{Circular}

Recall that a {\it circular ordering} on a set $S$ is a map $c: S\times S\times S \rightarrow \{ -1, 0, 1\}$ satisfying:
\begin{enumerate}
	\item If $(t_1, t_2, t_3) \in S^3$ then $c(t_1, t_2, t_3) = 0$ if and only if $\{t_1, t_2, t_3\}$ are not all distinct;
	\item For all $t_1, t_2, t_3, t_4 \in S$ we have
	\[ c(t_1, t_2, t_3) - c(t_1, t_2, t_4) + c(t_1, t_3, t_4)-c(t_2, t_3, t_4) = 0.
	\] 
\end{enumerate}
A set $S$ with a circular ordering is called {\it circularly orderable}. A group $G$ is called {\it left-circularly orderable} (respectively {\it right-circularly orderable}) if it admits a circular ordering $c: G\times G\times G \rightarrow \{ - 1, 0, 1\}$ such that $c(g_1,g_2, g_3)=c(gg_1,gg_2,gg_3)$ (respectively $c(g_1,g_2, g_3)=c(g_1g,g_2g,g_3g)$) for any $g, g_1, g_2$ and $g_3\in G$. In this case we say that the circular ordering $c$ is left-invariant (respectively right-invariant). 

The notion of circular orderability for quandles was defined in \cite{BE}. A quandle $Q$ is {\it left-circularly orderable} (respectively {\it right-circularly orderable}) if it admits a circular ordering $c: Q\times Q\times Q \rightarrow \{ - 1, 0, 1\}$ such that $c(q_1,q_2, q_3)=c(q*q_1,q*q_2,q*q_3)$ (respectively $c(q_1,q_2, q_3)=c(q_1*q,q_2*q,q_3*q)$) for any $q, q_1, q_2$ and $q_3\in Q$. 

Recall that if a quandle $Q$ is right (respectively left) orderable with a right or left ordering $\prec$, then it is also right (respectively left) circularly orderable \cite{BE} with a circular order $c_<$ defined by 
$$\;\;\;\; c_{\prec}(q_1,q_2,q_3)=\left\{\begin{array}{cl} 1 & \text{if $q_1\prec q_2\prec q_3$ or $q_2\prec q_3\prec q_1$ or $q_3\prec q_1\prec q_2$}
		\\ -1  &  \text{if $q_1\prec q_3\prec q_2$ or $q_3\prec q_2\prec q_1$ or $q_2\prec q_1\prec q_3$  }  
		\\ 0 & \text{otherwise.}
	\end{array}
	\right.$$ 
We call this circular ordering $c_{\prec}$ a {\it secret right (respectively left) ordering}.

Recall that the circle $\mathbb{S}^1$ is right orderable as a trivial quandle with the quandle conjugation operation. Therefore, it is also right circularly orderable. Below, we give uncountable many right and left circular orderings of the circle.

\begin{example}\label{S1}
	Fix a real number $t$ such that $0<t<1$ and consider the quandle operation on the circle $\mathbb{S}^1$ given by
	\[
	e^{i \lambda}*e^{i \mu}=e^{i[t \lambda +(1-t)\mu]}.
	\]
	Let $c$ be the map $c:\mathbb{S}^1 \times \mathbb{S}^1 \times \mathbb{S}^1 \rightarrow \{ - 1, 0, 1\}$ given by:
	$$(1)\;\;\;\; c(e^{i \theta},e^{i \phi},e^{i \psi})=\left\{\begin{array}{cl} 1 & \text{if $\theta <\phi<\psi$ or $\phi<\psi<\theta$ or $\psi<\theta<\phi $}
		\\ -1  &  \text{if $\theta < \psi< \phi$ or $\psi < \phi < \theta$ or $\phi < \theta < \psi$  }  
		\\ 0 & \text{otherwise}
	\end{array}
	\right.$$
	for any reals $\theta, \phi$ and $\psi\in [0, 2\pi)$.  A direct computation gives that $c$ is both right and left circular ordering on the quandle $\mathbb{S}^1$.  
\end{example}

Let $Q$ and $T$ be two quandles. An {\it action} of $Q$ on $T$ is a quandle homomorphism $$\varrho:Q\longrightarrow \mathrm{Conj}(\mathrm{Aut}(T))$$
where $\mathrm{Aut}(T)$ is the group of quandle automorphisms of $T$ and recall that the quandle operation on $\mathrm{Conj}(\mathrm{Aut}(T))$ is given by $f*g=gfg^{-1}$ for any $f, g \in \mathrm{Aut}(T)$.  We thus have $\varrho(x*y)=\varrho(y) \varrho(x) \varrho(y)^{-1}$ which equivalent to $\varrho(x*y) \varrho(y)=\varrho(y) \varrho(x)$. 

Let $A$ be non nonempty set. Then $A$ with the binary operation $a*b=a$ for any $a, b\in A$ is a trivial quandle. We then have that $\mathrm{Aut}(A)=S_A$ where $S_A$ the symmetric group on $A$, and this give the definition of the action of a quandle on a set $A$ ( see for example \cite{EM, EM2, RSS}).

Assume that $\varrho:Q\longrightarrow \mathrm{Conj}(\mathrm{Aut}(T))$ is an action of a quandle $Q$ on a quandle $T$.
For any $x\in T$, we call the set $Q_x:=\{p \in Q\;|\; \varrho(p)(x)=\varrho(q)(x),\; {\rm for\; some\;} q\in Q\setminus\{p\}\}$ the {\it stablizer} of $x$. 
\begin{theorem}
Let $Q$ be a quandle, then the following statement are true:
\begin{enumerate}
    \item If $Q$ is semi-latin and right circularly orderable, then it acts faithfully on a circularly ordered set $X$ by circular order preserving bijections;
    \item If $Q$ acts faithfully on $\mathbb{S}^1$ by right circular order preserving bijections and there exists a point $t\in \mathbb{S}^1$ such that the stabilizer $Q_t$ is empty then it is right circularly orderable.
    
\end{enumerate}
\end{theorem}
\begin{proof}
$(1)$ Assume that $Q$  is semi-latin and right circularly orderable with a right circular ordering $c$. We define the map 
$$\varrho:Q\longrightarrow \mathrm{Conj}(\mathrm{Aut}(Q))$$ by $\varrho(q)=R_q$ for any $q\in Q$. Since $R_q$ is an automorphism for any $q\in Q$, the map $\varrho$ is an action of $Q$ on $X=Q$. Let $q, x,y,z \in Q$.
 Then $c(\varrho(q)(x), \varrho(q)(y), \varrho(q)(z))=c(x*q, y*q, z*q)=c(x, y, z)$. Now, we need to show that $\varrho$ is injective. If $p, q \in Q$ such that $\varrho(p)=\varrho(q)$, then $x*p=x*q$ for any $x\in Q$, which implies that $p=q$ since the quandle is semi-latin.

$(2)$ Fix a right circular ordering $d$ on $\mathbb{S}^1$ and assume that $Q$ acts faithfully on $\mathbb{S}^1$ by right circular order preserving bijections.   Suppose there exists a point $t\in \mathbb{S}^1$ such that the stabilizer $Q_t$ is empty. Then define the map 
$c:Q\times Q\times Q\longrightarrow \{-1, 0, 1\}$ as follows,

$$\;\;\;\; c(x, y, z)=\left\{\begin{array}{cl} 1 & \text{if $d(\varrho(x)(t), \varrho(y)(t), \varrho(z)(t))=1 $}
	\\ -1  &  \text{if $d(\varrho(x)(t), \varrho(z)(t), \varrho(y)(t))=-1$  }  
	\\ 0 & \text{otherwise}
\end{array}
\right.$$

for any $x, y$ and $z\in Q$. Since the action $\varrho$ is faithful and the stabilizer set $Q_t$ is empty, then $x, y$ and $z\in Q$ are all distinct if and only if $\varrho(x)(t), \varrho(y)(t)$ and $\varrho(z)(t)\in \mathbb{S}^1$ are all distinct. Since $d$ is a right circular ordering on $\mathbb{S}^1$, the map $c$ is a right circular ordering on $Q$.
\end{proof}

\section{Circular orderability of link quandles}\label{Sec4}
For two distinct elements $x$ and $y$ of a quandle $(X,*)$ and a positive integer $i$, we denote $ y*_i x:=(R_x)^i(y)$.  Now we have the following lemma.

\begin{lemma}\label{triv}
	Let $(X,*)$ be a quandle.  Assume that the map $c:X \times X \times X \rightarrow \{-1,0,1\}$ is a right invariant circular ordering on $X$.  Let $x, y \in X$ be two distinct elements. 
 Then for any positive integer $i$, we have
	\[
	c(x,y,y*_i x)= c(x,y,y*_{i+1}x).
	\] 
\end{lemma}

\begin{proof}
The cocycle condition gives
$$c(x, y, y*x)-c(x, y, y*_2x)+c(x, y*x, y*_2x)-c(y, y*x, y*_2x)=0$$ which implies that
$$2c(x, y, y*x)-c(x, y, y*_2x)-c(y, y*x, y*_2x)=0.$$
Hence, we have that $c(x, y, y*x)=c(x, y, y*_2x)$.
By induction assume that $i>2$ and $$c(x, y, y*x)=c(x, y, y*_2x)=c(x, y, y*_3x)=\cdots =c(x, y, y*_ix).$$ The cocycle condition applied to the $4$-tuple $(x,y,y*_i x,y*_{i+1}x)$ gives the equation 
	\[
	c(x,y,y*_i x) -c(x,y,y*_{i+1} x) + c(x,y*_i x,y*_{i+1} x) -c(y,y*_i x,y*_{i+1} x)=0.
	\]
Therefore, using the right invariance of the map $c$, we can re-write this equation in the following form
		\[
	c(x,y,y*_i x) -c(x,y,y*_{i+1} x) + c(x,y,y*x) -c(y,y*_i x,y*_{i+1} x)=0.
	\]
	The induction hypothesis implies that 
	\[
	c(x,y,y*_i x) =c(x,y,y*_{i+1} x) = c(x,y,y*x) =c(y,y*_ix,y*_{i+1} x).
	\]
	This finishes the proof.
	
\end{proof}

\begin{theorem}\label{rcoinfinite}
    A nontrivial right circularly orderable quandle is infinite.
\end{theorem}
\begin{proof}
    Let $X$ be a nontrivial right circularly orderable quandle with a circularly ordering $c$. Let $x$ and $y$ be two distinct elements of $X$. Then by Lemma~\ref{triv}, for any positive integer $i$, we have that 
    $$c(x, y, y*x)=c(x, y, y*_2x)=c(x, y, y*_3x)=\cdots =c(x, y, y*_ix).$$ Since $x$ and $y$ are distinct and $X$ is nontrivial, we have 
    $$0\neq c(x, y, y*x)=c(x, y, y*_2x)=c(x, y, y*_3x)=\cdots =c(x, y, y*_ix).$$
    Therefore, the quandle $X$ is infinite.
    \end{proof}
    
Therefore, we recover a result in \cite{BPS2}, \cite{DDHPV} and \cite{RSS} as:

\begin{corollary}
    A nontrivial right orderable quandle is infinite.
\end{corollary}

Recall \cite{Joyce} that an $n$-quandle is a quandle $(X,*)$ such that for all $x \in X, (R_x)^n$ is the identity map.  

\begin{theorem}\label{n-quandle}
	Let $n\geq 2$ be a positive integer and let $(X,*)$ be a $n$-quandle, then $X$ cannot be right circular orderable.
\end{theorem}
\begin{proof}  Assume that $(X,*)$ is a $n$-quandle.  By Lemma~\ref{triv}, we have the equality 
			\[
		c(x,y,y*x)= \dots= c(x,y,y*_{n-1}x) =c(x,y,y*_n x) =  c(x,y,y)=0.		\]
		This gives a contradiction.  Thus the $n$-quandle  $X$ cannot be right circular orderable.
\end{proof}

\begin{lemma}\label{lefttriv}
	Let $(X,*)$ be a quandle.  Assume that the map $c:X \times X \times X \rightarrow \{-1,0,1\}$ is a left invariant circular ordering on $X$.  Let $x, y \in X$ be two distinct elements. 
 Then for any positive integer $i$, we have
	\[
	c(x,y, L_x^i(y))= c(x,y,L_x^{i+1}(y)).
	\] 
\end{lemma}

\begin{proof}
We have that 
$$c(x, y, x*y)-c(x, y, L_x^2(y))+c(x, x*y, L_x^2(y))-c(y, x*y, L_x^2(y))=0$$ which implies that
$$2c(x, y, x*y)-c(x, y, L_x^2(y))-c(y, x*y, L_x^2(y))=0.$$
Hence, we have that $c(x, y, x*y)=c(x, y, L_x^2(y))$.
By induction assume that $i>2$ and $$c(x, y, x*y)=c(x, y, L_x^2(y))=c(x, y, L_x^3(y))=\cdots =c(x, y, L_x^i(y)).$$ The cocycle condition applied to the $4$-tuple $(x,y,L_x^i(y),L_x^{i+1}(y))$ gives the equation 
	\[
	c(x,y,L_x^i(y)) -c(x,y,L_x^{i+1}(y)) + c(x,L_x^i(y), L_x^{i+1}(y)) -c(y, L_x^i(y), L_x^{i+1}(y))=0.
	\]
Therefore, using the left invariance of the map $c$, we can re-write this equation in the following form
		\[
	c(x,y,L_x^i(y)) -c(x,y,L_x^{i+1}(y)) + c(x,y,x*y) -c(y,L_x^i(y), L_x^{i+1}(y))=0.
	\]
	The induction hypothesis implies that 
	\[
	c(x,y, L_x^i(y)) =c(x,y, L_x^{i+1}(y)) = c(x,y, x*y) =c(y, L_x^i(x), L_x^{i+1}(y)).
	\]
	This finishes the proof.
	
\end{proof}

\begin{theorem}\label{lcoinfinite}
    A nontrivial left circularly orderable quandle is semi-latin and infinite.
\end{theorem}
\begin{proof}
    Let $X$ be a nontrivial left circularly orderable quandle with a left circularly ordering $c$. 
    Let $x, y$ and $z$ be three distinct elements of $X$ such that $x*y=x*z$. Then we have $0\neq c(x, y, z)=c(x, x*y, x*z)=0$ which is a contraction. So, the quandle $X$ is semi-latin.
    Let $x$ and $y$ be two distinct elements of $X$.
    Then by Lemma~\ref{lefttriv}, for any positive integer $i$, we have that 
    $$c(x, y, x*y)=c(x, y, L_x^2(y))=c(x, y, L_x^3(y))=\cdots =c(x, y, L_x^i(y)).$$     
     Since $x$ and $y$ are distinct and $X$ is nontrivial, we have 
$$0\neq c(x, y, x*y)=c(x, y, L_x^2(y))=c(x, y, L_x^3(y))=\cdots =c(x, y, L_x^i(y)).$$
    Therefore, the quandle $X$ is infinite.
    \end{proof}

Therefore, we recover a result in \cite{BPS2} and \cite{RSS} as:

\begin{corollary}
    A nontrivial left orderable quandle is semi-latin and infinite.
\end{corollary}

\begin{remark} By Theorem \ref{n-quandle}, 
	involutive quandles are not \emph{right circularly} orderable. Therefore, dihedral quandles are not circularly orderable. In \cite{RSS} it was shown that the fundamental quandle $Q(4_1)$ of the figure eight knot is right orderable.  Then from Lemma 3.1 of \cite{BE} we deduce that $Q(4_1)$ is right circularly orderable. Since $Q_2(4_1)$ is isomorphic to the dihedral quandle $R_5$, we can conclude that a homomorphic image of \emph{right circularly} orderable quandle may not be  \emph{right circularly} orderable.
\end{remark}

\section{Quandle extensions by $2$-cocycles and orderability}\label{Sec5}

Let $Q$ be a quandle and $S$ be a nonempty set. Define a binary operation $(s,x)*(t,y)=(\alpha_{x,y}(s,t), x*y)$ on $S\times Q$ where 
$\alpha :Q\times Q\longrightarrow \text{Fun}(S\times S, S)=S^{S\times S}$ is a function. Then, $S\times Q$ is a quandle if and only if $\alpha$ satisfies the following conditions:
\begin{enumerate}
    \item $\alpha_{x,x}(s,s)=s$ for all $x\in Q$
 and $s\in S$;
 \item $\alpha_{x,y}(- , s): S\longrightarrow S$ is a bijection for all $x,y\in Q$ and $s\in S$;
 \item $\alpha_{x*y,z}(\alpha_{x, y}(r,s), t)=\alpha_{x*z, y*z}(\alpha_{x, z}(r,t), \alpha_{y, z}(s, t))$ for all $x, y$ and $z\in Q$, and for all $r, s$ and $t\in S$.
 \end{enumerate}
 
The function $\alpha$ is called a {\it dynamical quandle cocycle}. This type of quandles constructed above are called {\it quandle extensions} of $Q$ with respect to the dynamical cocycle $\alpha$, and are denoted by $S\times_{\alpha}Q$ (see \cite{AG} and \cite{CEGS}).

In particular, assume that $S=A$ is an abelian group and that for any $ x, y \in Q,$ $\alpha_{x, y}(a,b)=\eta_{x,y}(a)+\tau_{x,y}(b) + \kappa_{x, y}$ for any $(a, b)\in A\times A$, where the map $\eta_{x,y}:A \rightarrow A$ is an automorphism of $A$, $\tau_{x,y}:A \rightarrow A$ is an endomorphism and $\kappa: Q\times Q\rightarrow A$ is a generalized $2$-cocycle.  Then it is known \cite{AG, CEGS} that $\alpha$ is dynamical cocycle if and only if $\eta$ and $\tau$ satisfy the following conditions, for all $x,y,z \in Q$, 
\begin{eqnarray}
	\eta_{x*y,z}\eta_{x,y} &=&  \eta_{x*z,y*z}\eta_{x,z},\label{eq1}\\
	\eta_{x*y,z}\tau_{x,y} &=&  \tau_{x*z,y*z}\eta_{y,z},\label{eq2}\\
	\tau_{x*y,z}&=& \eta_{x*z,y*z} \tau_{x,z}+\tau_{x*z,y*z}\tau_{y,z},\label{eq3}\\
	\tau_{x,x} + \eta_{x,x} &=&Id.\label{eq4}
	\end{eqnarray}
	
Let $<$ and $<'$ be two total orderings on $Q$ and $S$ respectively. Then the {\it lexicographic ordering} $\prec$ on $S\times_{\alpha}Q$ is defined as follows, for any $(r,x)$ and $(s,y)\in S\times Q$, $(r,x)\prec(s,y)$ if and only if $r<'s$ or $r=s$ and $x<y$.	
	
	\begin{definition}
		Let $(X,*)$ be quandle.  A \emph{module} $(A,\eta,\tau)$ over $X$ comprise of:
		\begin{itemize}
			\item 
			An abelian group $A_x$ for each element $x \in X$.
			\item
			An isomorphism $\eta_{x,y}:A_x \rightarrow A_{x*y}$ for any elements $x,y$ of $X$.
			\item
			An endomorphism $\tau_{x,y}:A_y \rightarrow A_{x*y}$ for any elements $x,y$ of $X$,	
		\end{itemize} 
		such that equations (\ref{eq1}), (\ref{eq2}), (\ref{eq3}) and  (\ref{eq4}) hold.  This can be also expressed by the following diagrams.

			\begin{eqnarray*}
		\xymatrix{
			A_x \ar[r]^{ \eta_{x,y}}  \ar[d]_ { \eta_{x,z} } & A_{x*y} \ar[d]^{ \eta_{x*y,z} }
			\\ A_{x*z} \ar[r]_ { \eta_{x*z,y*z} } & A_{(x*y)*z}
		}
		\end{eqnarray*}
		
			\begin{eqnarray*}
		\xymatrix{
			A_y \ar[r]^{ \tau_{x,y}}  \ar[d]_ { \eta_{y,z} } & A_{x*y} \ar[d]^{ \eta_{x*y,z} }
			\\ A_{x*z} \ar[r]_ { \tau_{x*z,y*z} } & A_{(x*y)*z}
		}
		\end{eqnarray*}
		
		 	\begin{eqnarray*}
		 	\xymatrix{
		 		A_z 	\ar[rrrr]^{\tau_{x*y,z}} \ar[dddd]_{\tau_{x,z}} \ar[rrdd]^{\tau_{y,z}} &&&& A_{(x*y) * z} \ar[dddd]^{Id}\\
		 		&&&& \\
		 		& & A_{y*z} \ar[rrdd]^{\tau_{x*z,y*z}} && \\
		 		&&&& \\
		 		A_{x*z} \ar[rrrr]^{\eta_{x*z,y*z}} &&&& A_{(x*z)*(y*z)}
		 	}
		 \end{eqnarray*}
	\end{definition}

\begin{remark}
	Notice that the abelian group $A_x$ corresponds to the orbit of $x$ under the action of the inner group $Inn(x)$, that is, if $x$ and $y$ are in the same orbit then $A_x$ is isomorphic to $A_y$.  For example, if the quandle $X$ is connected (only one orbit), then for any $x,y \in X$, we have $A_x \cong A_y$ and then we will denote $A_x$ by just $A$.
\end{remark}

\begin{example}
	Let $\Lambda = \mathbb{Z}[t,t^{-1}]$ denote the ring of Laurent polynomials. 
	Then any  $\Lambda$-module $M$ is an $X$-module for any quandle $X$,
	by $\eta_{x,y} (a)=ta$ and $\tau_{x,y} (b) = (1-t)b $
	for any $x,y \in X$.
\end{example}

\begin{proposition}
Let $Q$ be a quandle and $(A, <')$ be a right orderable subgroup of $\mathbb{R}$. If there exists a nonzero element $t$ of $A$ such that for any $ x, y \in Q,$ $\alpha_{x, y}(a, b)=ta+(1-t)b$ for any $a, b\in A$, is a dynamical cocycle, then 

\begin{enumerate}
    \item $A\times_{\alpha}Q$ is a right orderable quandle if $Q$ is right orderable and $t>0$.
    \item $A\times_{\alpha}Q$ is a left orderable quandle if $Q$ is left orderable and $t<1$.
\end{enumerate}
\end{proposition}
\begin{proof}
$(1)$ Assume that $Q$ admits a right ordering $<$ and $t>0$. Let $(a, x), (b, y)$ and $(c,z)\in A\times Q$ such that $(a,x)\prec (b, y)$. Then $(a,x)*(c,z)=(\alpha_{x,z}(a,c), x*z)=(ta+(1-t)c, x*z)$ and $(tb+(1-t)c, y*z)=(\alpha_{y,z}(b,c), y*z)=(b,y)*(c,z)$. We have two cases.

{\bf Case 1:} Assume $a<'b$. Then, since $A$ is Abelian, the ordering $<'$ is both left and right orderable, so $ta+(1-t)c<'tb+(1-t)c$ which implies $(a,x)*(c,z)=(\alpha_{x,z}(a,c), x*z)=(ta+(1-t)c, x*z)\prec (tb+(1-t)c, y*z)=(\alpha_{y,z}(b,c), y*z)=(b,y)*(c,z)$.

{\bf Case 2:} Assume $a=b$ and $x<y$. Then $x*z<y*z$ which implies $(a,x)*(c,z)=(\alpha_{x,z}(a,c), x*z)=(ta+(1-t)c, x*z)\prec (ta+(1-t)c, y*z)=(tb+(1-t)c, y*z)=(\alpha_{y,z}(b,c), y*z)=(b,y)*(c,z)$.

This complete the proof of the fact that $A\times_{\alpha} Q$ is right orderable.

$(2)$ A similar proof show that $A\times_{\alpha}Q$ is a left orderable quandle if $Q$ is left orderable and $t<1$.
\end{proof}

\begin{corollary}
Let $Q$ be a quandle and $(A, <')$ be a right orderable subgroup of $\mathbb{R}$. Assume there exists a nonzero element $t$ of $A$ such that for any $ x, y \in Q,$ $\alpha_{x, y}(a, b)=ta+(1-t)b+\kappa_{x, y}$ for any $a, b\in A$ and $x, y\in Q$, is a dynamical cocycle. Then, 

\begin{enumerate}
    \item $A\times_{\alpha}Q$ is a right orderable quandle if $Q$ is right orderable, $t>0$ and $\kappa_{x, z}=\kappa_{y, z}$ for any $x, y$ and $z\in Q$.
    \item $A\times_{\alpha}Q$ is a left orderable quandle if $Q$ is left orderable, $t<1$ and $\kappa_{z, x}=\kappa_{z, y}$ for any $x, y$ and $z\in Q$.
\end{enumerate}
\end{corollary}

\begin{proposition}
Let $Q$ be a quandle and $(A, <)$ be a right orderable subgroup of $\mathbb{R}$. If there exists a nonzero element $t$ of $A$ such that for any $ x, y \in Q,$ $\alpha_{x, y}(a, b)=ta+(1-t)b$ for any $a, b\in A$, is a dynamical cocycle, then 
\begin{enumerate}
    \item $A\times_{\alpha}Q$ is a right circularly orderable quandle if $Q$ is right circularly orderable and $t>0$.
    \item $A\times_{\alpha}Q$ is a left circularly orderable quandle if $Q$ is left circularly orderable and $t<1$.
\end{enumerate}
\end{proposition}
\begin{proof} $(1)$ Assume that $d$
 is a right circular ordering on $Q$ and $t>0$. Define a map $c:(A\times Q)^3\longrightarrow\{-1, 0, 1\}$ as follows: 
let $(a, x), (b, y)$ and $(r, z)$ be three elements of $A\times Q$.
\begin{itemize}
    \item Assume that the three elements $(a, x), (b, y)$ and $(r, z)$ are all distinct.
    \begin{enumerate}
    \item if $x$, $y$ and $z$ are all distinct, then $c((a, x), (b, y), (r, z))=d(x, y, z)$;
    \item if exactly two of the elements $x, y$ and $z$ are the same, without loss of generality and up to cyclic permutation assume that $x=y$, then $c((a, x), (b, y), (r, z))=1$ if $a<b$ and $c((a, x), (b, y), (r, z))=-1$ otherwise;
    \item if $x=y=z$ then $c((a, x), (b, y), (r, z))=c_<(a, b, r)$.
\end{enumerate}
\item If at least two of the elements $(a, x), (b, y)$ and $(r, z)$ are the same then $$c((a, x), (b, y), (r, z))=0.$$
\end{itemize}

Now let $(a, x), (b, y),$ $(r, z)$ and $(s, q)\in A\times Q$ be three distinct elements, then we have four 
 cases:

{\bf Case 1:} If $x, y, z$ and $q$ are all distinct, then
$c((a, x),(b, y), (r,z))-c((a,x),(b,y),(s,q))+ c((a,x), (r,z), (s,q))-c((b,y), (r,z), (s, q))=d(x, y, z)-d(x, y, q)+ d(x, z,   q)-d(y, z, q)=0.$ We have also that, $c((a, x)*(s, q),(b, y)*(s, q), (r,z)*(s, q))=c((ta+(1-t)s, x*q), (tb+(1-t)s, y*q), (tr+(1-t)s, z*q))=d(x*q, y*q, z*q)=d(x, y, z)=c((a, x),(b, y), (r,z))$.

{\bf Case 2:} If exactly two of $x, y, z$ and $q$ are the same. Without loss of generality assume that $x=y$. Then,
$c((a, x),(b, y), (r,z))-c((a,x),(b,y),(s,q))+ c((a,x), (r,z), (s,q))-c((b,y), (r,z), (s, q))=c((a, x),(b, y), (r,z))-c((a,x),(b,y),(s,q))+ d(x, z, q)-d(y, z, q)=c((a, x),(b, y), (r,z))-c((a,x),(b,y),(s,q))$. We have that two subcases:

{Subcase 1:} Assume that $a<b$. Then $c((a, x),(b, y), (r,z))-c((a,x),(b,y),(s,q))=1-1=0$. We have also that, $c((a, x)*(s, q),(b, y)*(s, q), (r,z)*(s, q))=c((ta+(1-t)s, x*q), (tb+(1-t)s, y*q), (tr+(1-t)s, z*q))=1=c((a, x),(b, y), (r,z))$.

{Subcase 2:} Assume that $b<a$. Then $c((a, x),(b, y), (r,z))-c((a,x),(b,y),(s,q))=-1+1=0$. We have also that, $c((a, x)*(s, q),(b, y)*(s, q), (r,z)*(s, q))=c((ta+(1-t)s, x*q), (tb+(1-t)s, y*q), (tr+(1-t)s, z*q))=-1=c((a, x),(b, y), (r,z))$.

{\bf Case 3:} If exactly three of $x, y, z$ and $q$ are the same. Without loss of generality assume that $x=y=z\neq q$. Then, $c((a, x)*(s, q),(b, y)*(s, q), (r,z)*(s, q))=c((ta+(1-t)s, x*q), (tb+(1-t)s, y*q), (tr+(1-t)s, z*q))=c_<(ta+(1-t)s, tb+(1-t)s, tr+(1-t)s)=1=c_<(a, b, r)=c((a, x),(b, y), (r,z))$. We have also that,

$\epsilon:=c((a, x),(b, y), (r,z))-c((a,x),(b,y),(s,q))+ c((a,x), (r,z), (s,q))\\ -c((b,y), (r,z), (s, q))=c_<(a, b, r)-c((a,x),(b,y),(s,q))+ c((a,x), (r,z), (s,q))-c((b,y), (r,z), (s, q))$. We have two subcases:

{Subcase 1:} Up to cyclic permutation we can assume that $a<b<r$. Then 
\begin{align*}
   \epsilon&= c_<(a, b, r)-c((a,x),(b,y),(s,q))+ c((a,x), (r,z), (s,q))-c((b,y), (r,z), (s, q))\\&=1-1+1-1=0. 
\end{align*}

{Subcase 2:} Up to cyclic permutation we can assume that $a<r<b$. Then 
\begin{align*}
   \epsilon&= c_<(a, b, r)-c((a,x),(b,y),(s,q))+ c((a,x), (r,z), (s,q))-c((b,y), (r,z), (s, q))\\&=-1-1+1+1=0. 
\end{align*}

{\bf Case 4:} If $x=y=z=q$ then $c((a, x),(b, y), (r,z))-c((a,x),(b,y),(s,q))+\\ c((a,x), (r,z), (s,q))-c((b,y), (r,z), (s, q))=c_<(a, b, r)-c_<(a, b, s)+ c_<(a, r,   s)-c_<(b, r, s)=0.$
This complete the proof to the fact that $A\times_{\alpha}Q$ is a right circularly orderable quandle if $Q$ is right circularly orderable and $t>0$.

(2) A similar proof show that $A\times_{\alpha}Q$ is a left circularly orderable quandle if $Q$ is left circularly orderable and $t<1$.
\end{proof}

\begin{corollary}
Let $Q$ be a quandle and $(A, <)$ be a right orderable subgroup of $\mathbb{R}$. Assume there exists a nonzero element $t$ of $A$ such that for any $ x, y \in Q,$ $\alpha_{x, y}(a, b)=ta+(1-t)b+\kappa_{x, y}$ for any $a, b\in A$ and $x, y\in Q$, is a dynamical cocycle. Then
\begin{enumerate}
    \item $A\times_{\alpha}Q$ is a right circularly orderable quandle if $Q$ is right circularly orderable, $t>0$ and $\kappa_{x, z}=\kappa_{y, z}$ for any $x, y$ and $z\in Q$.
    \item $A\times_{\alpha}Q$ is a left circularly orderable quandle if $Q$ is left circularly orderable, $t<1$ and $\kappa_{z, x}=\kappa_{z, y}$ for any $x, y$ and $z\in Q$.
\end{enumerate}
\end{corollary}

\begin{lemma}
If the lexicographic ordering $\prec$ on $S\times_{\alpha}Q$ is right-invariant $($respectively left-invariant$)$ then, for all $x, y\in Q$ and $s\in S$ , the map $\alpha_{x,y}(- , s): S\longrightarrow S$ $($respectively $\alpha_{x,y}(s, -): S\longrightarrow S)$ is an order preserving map on $S$.
\end{lemma}
\begin{proof}
Assume that the lexicographic ordering $\prec$  on $S\times_{\alpha}Q$ is right-invariant. Let $x, y\in Q$ and $s_0\in S$. Let $r, s\in S$ such that $r<'s$. Then $(r,x)*(s_0,y)=(\alpha_{x,y}(r,s_0), x*y)\prec (s,x)*(s_0,y)=(\alpha_{x,y}(s,s_0), x*y)$ which implies that $\alpha_{x,y}(r,s_0)<'\alpha_{x,y}(s,s_0)$. Therefore, for all $x,y\in Q$ and $s\in S$, the map $\alpha_{x,y}(- , s): S\longrightarrow S$  is an order preserving bijection on $S$. Similarly, if the lexicographic ordering $\prec$  on $S\times_{\alpha}Q$ is left-invariant then the map $\alpha_{x,y}(s, -): S\longrightarrow S$ is an order preserving map on $S$.
\end{proof}

\begin{proposition} Let $(S,<')$ be a totally ordered set and $(Q, <)$ be a right orderable quandle. Let $S\times_{\alpha}Q$ be a quandle extension associated with some dynamical quandle cocycle $\alpha$. If the map $\alpha_{-,y}(- , s): S\times Q\longrightarrow S$ for all $y\in Q$ and $s\in S$ is an order preserving map with respect to $\prec$, then the lexicographic ordering $\prec$ on $S\times_{\alpha}Q$ is a right ordering.
\end{proposition}
\begin{proof}

Assume that the map $\alpha_{-,y}(- , s): S\times Q\longrightarrow S$ for all $y\in Q$ and $s\in S$ is an order preserving map with respect to $\prec$. Let $(r, x), (s, y)$ and $(t,z)\in S\times Q$ such that $(r,x)\prec (s, y)$. Then $\alpha_{x,z}(r,t)<' \alpha_{y,z}(s,t)$ which implies $(r,x)*(t,z)=(\alpha_{x,z}(r,t), x*z)\prec(\alpha_{y,z}(s,t), y*z)=(s,y)*(t,z)$.
\end{proof}

\section{Group extensions versus quandle Extensions and orderability}\label{Sec6}
Let $0\longrightarrow A \longrightarrow G\longrightarrow K\longrightarrow 1$ be a short exact sequence of groups, where $A$ is a $K$-module. If $A$ and $K$ are left-orderable then $G$ is left-orderable (see \cite{BaC1} and \cite{Cal}). If $A$ is left orderable and $K$ is left circularly orderable then $G$ is left circularly orderable (see \cite{BaC1} and \cite{Cal}). Now consider $G$ and $K$ as quandles, $\mathrm{Conj}(G)$ and $\mathrm{Conj}(K)$, respectively. Then $G$ is an extension of $K$ by the dynamical cocycle $\alpha:K\times K \longrightarrow A^{A\times A}$ define by $\alpha_{x, y}(a,b)=\eta_{x,y}(a)+\tau_{x,y}(b) + \kappa_{x, y}$ Where $\eta_{x,y}(a)=ya$,
$\tau_{x,y}(b)=b-(x*y)b$ and $\kappa_{x, y}=\theta(y, x)-yx\theta(y^{-1}, y)+\theta(yx, y^{-1})$ for any $x, y\in K$ and $a, b\in A$. Here $\theta$ is a group $2$-cocycle such that $G$ is the twisted semi-direct product $A\rtimes_{\theta}K$ (see \cite{CEGS}). We have the following two results.

\begin{proposition}\label{proplo1}
If the groups $A$ is right orderable, $K$ is bi-orderable and the ordering of $A$ is invariant under conjugation by elements of $G$, then the conjugation quandle of $G$ is right orderable.
\end{proposition} 
\begin{proof} 
Assume that the groups $A$ is right orderable, $K$ is bi-orderable and the ordering of $A$ is invariant under conjugation by elements of $G$. Then, the group $G$ is bi-orderable by \cite{PR}. Therefore $G$ is right orderable as the conjugation quandle by \cite{DDHPV, BPS}.
\end{proof}

Let $\prec$ be the lexicographic order on $A\times K$. We have the following:
\begin{proposition}\label{proplo2}
 Assume that the action of $K$ on $A$ is order preserving and $\kappa_{x, z}=\kappa_{y, z}$ for any $x, y$ and $z\in K$. Then $G$ is a right orderable quandle,
 \begin{enumerate}
     \item if $(A, <')$ and $(K, <)$ are right orderable quandles; and
     \item $z(a-b)+(y*z)c-(x*z)c<' 0$ whenever $(a, x)\prec (b, y)$ for any $a, b, c \in A$ and $x, y, z \in K$.
 \end{enumerate}
\end{proposition}
\begin{proof}
Assume that $A$ and $K$ admit right orderings $<'$ and $<$ respectively. Recall that $\alpha_{x, y}(a,b)=\eta_{x,y}(a)+\tau_{x,y}(b) + \kappa_{x, y}$ where $\eta_{x,y}(a)=ya$,
	$\tau_{x,y}(b)=b-(x*y)b$.
Let $(a, x), (b, y)$ and $(c,z)\in A\times K$ such that $(a,x)\prec (b, y)$. Then $(a,x)*(c,z)=(\alpha_{x,z}(a,c), x*z)=(za+c-(x*z)c+\kappa_{x, z}, x*z)$ and $(zb+c-(y*z)c + \kappa_{y, z}, y*z)=(\alpha_{y,z}(b,c), y*z)=(b,y)*(c,z)$. We have two cases.

{\bf Case 1:} Assume $a<'b$. Then since $A$ is Abelian, the ordering $<'$ is both left and right invariant, so $za+c-(x*z)c+\kappa_{x, z}<'zb+c-(y*z)c+\kappa_{y, z}$ which implies $(a,x)*(c,z)=(\alpha_{x,z}(a,c), x v*z)= (za+c-(x*z)c+\kappa_{x, z}, x*z)\prec (zb+c-(y*z)c+\kappa_{y, z}, y*z)=(\alpha_{y,z}(b,c), y*z)=(b,y)*(c,z)$.

{\bf Case 2:} Assume $a=b$ and $x<y$. Then $x*z<y*z$ which implies $(a,x)*(c,z)=(\alpha_{x,z}(a,c), x*z)= (za+c-(x*z)c+\kappa_{x, z}, x*z)\prec (zb+c-(y*z)c+\kappa_{y, z}, y*z)=(\alpha_{y,z}(b,c), y*z)=(b,y)*(c,z)$.

This complete the proof to the fact that $A\times_{\alpha} K$ is right orderable.
\end{proof}

\begin{proposition}\label{proplo3}
Assume that the action of $K$ on $A$ is order preserving and $\kappa_{z, x}=\kappa_{z, y}$ for any $(x, z)$ and $(y, z)\in K\times K$. Then $G$ is a left orderable quandle,
 \begin{enumerate}
     \item if $(A, <')$ and $(K, <)$ are left orderable quandles; and
     \item $a-b+(z*y)b-(z*x)a<' 0$ whenever $(a, x)\prec (b, y)$ for any $a, b, c \in A$ and $x, y, z \in K$.
 \end{enumerate}
\end{proposition}
\begin{proof}
 A similar proof as the one of the previous proposition show this proposition.
\end{proof}

\section{Cyclic branched covers and $n$-quandles}\label{Sec7}
In this section, all links are tame links in $\mathbb{S}^3$.
Let $L\subset \mathbb{S}^3$ be an $m$-components oriented link. Let $X$ be the exterior of $L$ in $\mathbb{S}^3$. Then, there is an epimorphism $\pi_1(X)\longrightarrow \mathbb{Z}$ which sends the meridians of $L$ oriented positively with respect to the orientation of $L$ to the generator of $\mathbb{Z}$. The covering $X_{\infty}\longrightarrow X$ associated to this epimorphism is  called the infinite cyclic covering. By reduction (mod $n$), we get a map $\pi_1(X)\longrightarrow \mathbb{Z}_n$ which determines the $n$-fold cyclic cover $q:X_n\longrightarrow X$ of $L$. The $n$-fold cyclic branched cover $\Sigma_n(L)$ is obtained by gluing $m$-solid tori to $X_n$ by identifying preimage of meridians of $L$ by $q$ to the meridians of the solid tori.

Given a quandle $Q$, the {\it enveloping group} $\mathrm{Env}(Q)$ of $Q$ is the group defined by,

$$\mathrm{Env}(Q)=\langle \widetilde{q}\; {\rm for \; all\;}q\in Q\; |\; \widetilde{p*q}=\widetilde{q}^{-1}\widetilde{p} \;\widetilde{q}\; {\rm for\; all\;} p\in Q\rangle ,$$ and for any $n\geq 2$, the {\it $n$-enveloping group} $\mathrm{Env}_n(Q)$ of $Q$ is defined by, $$\mathrm{Env}_n(Q)=\langle \widetilde{q}\; {\rm for \; all\;}q\in Q\; |\; \widetilde{p*q}=\widetilde{q}^{-1}\widetilde{p} \; \widetilde{q}\;{\rm and}\; \widetilde{q}^n=1\; {\rm for\; all\;} p\in Q\rangle .$$ 

This notion of enveloping group of a quandle $Q$ was defined in \cite{Joyce} as $\mathrm{Adconj}(Q)$, and $\mathrm{Conj}(Q)$ in \cite{Winker}.
As examples, if $Q$ is the trivial quandle then the enveloping group of $Q$ is the free abelian group with rank equal to the cardinality of $Q$. We have also the fact that the enveloping group of a link quandle is the link group (see \cite{Joyce}).

Let $\mathrm{Env}(Q)$ and $\mathrm{Env}_n(Q)$ be the enveloping group and $n$-enveloping group of a quandle $Q$ respectively. The {\it exponent} of an element $\widetilde{q_1}^{e_1}\widetilde{q_2}^{e_2}\cdots \widetilde{q_k}^{e_k}$ of $\mathrm{Env}(Q)$ written as a product of generators is the summation $e_1+e_2+\cdots e_k$. The {\it exponent} of an element $\widetilde{q_1}^{e_1}\widetilde{q_2}^{e_2}\cdots \widetilde{q_k}^{e_k}$ of $\mathrm{Env}_n(Q)$ written as a product of generators is the summation $e_1+e_2+\cdots e_k$ modulo $n$. The {\it exponent zero subgroup} $E^0(Q)$ (respectively $E^0_n(Q)$) of $\mathrm{Env}(Q)$ (respectively of $\mathrm{Env}_n(Q)$) is the set of all elements with exponent equal to zero. 

Let $L$ be a link in $\mathbb{S}^3$. Then the fundamental group of the $n$-fold cyclic branched cover $\pi_1(\Sigma_n(L))$ is equal to
the {\it exponent zero subgroup} $E^0_n(Q(L))\cong E^0_n(Q_n(L))$ of $\mathrm{Env}_n(Q(L))$, where $Q(L)$ is the fundamental quandle of $L$ (see \cite{Winker}) and $Q_n(L)$ is the $n$-quandle of $Q(L)$. Notice also that the index of $E^0_n(Q(L))\cong E^0_n(Q_n(L))$ in $\mathrm{Env}_n(Q(L))$ is equal to $n$.

Let $L$ be a link in $\mathbb{S}^3$. By Theorem \ref{n-quandle}, for any positive integer $n$, the $n$-quandle $Q_n(L)$ is not right circularly orderable, hence it is also not right orderable. In particular, the involutory quandle $Q_2(L)$ is not right circularly orderable and hence not right orderable. On the other hand, there are infinitely many links for which their $2$-enveloping group is right circularly orderable.

\begin{theorem}
    There are infinitely many links $L$ for which $\mathrm{Env}_2(Q(L))$ is right circularly orderable.
\end{theorem}
\begin{proof}
Let $L=T(p, q)\subset\mathbb{S}^3$ be a $(p,q)$-torus knot with $p>3$ and $q>5$. Then $\pi_1(\Sigma_2(L))$ is right 
orderable by Theorem 1.2 of \cite{GL}. Since $\pi_1(\Sigma_2(L))$ is an index 2 subgroup of $\mathrm{Env}_2(Q(L))$, there exists a short exact sequence
$1\longrightarrow \pi_1(\Sigma_2(L)) \longrightarrow \mathrm{Env}_2(Q(L)) \longrightarrow \mathbb{Z}_2\longrightarrow 1$. Since $\pi_1(\Sigma_2(L))$ is right orderable and $\mathbb{Z}_2$ is right circular orderable, $\mathrm{Env}_2(Q(L))$ is right circular orderable by a lexicographic circular ordering.
\end{proof}

It was shown in \cite{Hu}, Theorem 4.3 on page 401,  that for any $(p,q)$ two-bridge knot, with $p \equiv 3$ modulo $4$,  there
are only finitely many cyclic branched covers whose fundamental groups are not
right-orderable.  Thus we have the following proposition giving  infinitely many links for which the $n$-enveloping group of their fundamental quandle $Q(L)$ is right circularly orderable for $n$ sufficiently large prime number.

\begin{proposition}\label{H}
	
	Let $L$ be a $(p,q)$ two-bridge knot  , with $p \equiv 3$ modulo $4$.  For $n$ sufficiently large \emph{prime number},  the group $\mathrm{Env}_n(Q(L))$  is right circularly orderable.
\end{proposition}
\begin{proof}
	Assume that $p \equiv 3$ modulo $4$.  There are infinitely many $(p,q)$ two-bridge knots $L$ for which $\pi_1(\Sigma_n(L))$ is right orderable group \cite{Hu}.  Since $n$ is prime, then there is a short exact sequence
	$1\longrightarrow \pi_1(\Sigma_n(L)) \longrightarrow \mathrm{Env}_n(Q(L)) \longrightarrow \mathbb{Z}_n\longrightarrow 1$. Since $\pi_1(\Sigma_n(L))$ is right orderable \cite{Hu} and $\mathbb{Z}_n$ is right circular orderable, $\mathrm{Env}_n(Q(L))$ is right circular orderable by a lexicographic circular ordering.
\end{proof}

\section{Non-bi-orderability of knot group}\label{Sec8}

Until now, it is not known exactly which links in the $3$-sphere $\mathbb{S}^3$ have bi-orderable link group. In this section, we investigate this question, ask a generalized question, and give an answer to Question 1.1 in \cite{CDN}. We have the following:
\begin{question}
    Give a Characterization of links in $\mathbb{S}^3$ that have non-bi-circularly orderable (resp. non-bi-orderable) link group.
\end{question}

Given a quandle $Q$, the {\it enveloping group} $\mathrm{Env}(Q)$ of $Q$ is the group defined by,

$$\mathrm{Env}(Q)=\langle \widetilde{q}\; {\rm for \; all\;}q\in Q\; |\; \widetilde{p*q}=\widetilde{q}^{-1}\widetilde{p} \; \widetilde{q}\; {\rm for\; all\;} p\in Q\rangle.$$ 

 If $Q$ is a trivial quandle, then $\mathrm{Env}(Q)$ is the free abelian group of rank equal to the cardinality of $Q$. We have the following result of Joyce:

\begin{theorem}\cite{Joyce}\label{isomp} Let $L$ be an oriented link in $\mathbb{S}^3$. The enveloping group $\mathrm{Env}(Q(L))$ is isomorphic to the link group $\pi_1(\mathbb{S}^3\setminus L)$ of $L$.    
\end{theorem}

Let $Q$ be a quandle. There is a natural map $\rho: Q \longrightarrow \mathrm{Env}(Q)$ defined by $\rho(q)=\widetilde{q}$. Moreover, the induced map $\widetilde{\rho}:Q\longrightarrow \mathrm{Conj}(\mathrm{Env}(Q))$ defined by $\rho(q)=\widetilde{q}$ is a quandle homomorphism (see \cite{Winker}). We have the following result of Ryder:

\begin{theorem}\cite{Ryder}\label{prime}
    Let $L$ be an oriented link in $\mathbb{S}^3$. The quandle homomorphism 
\begin{align*}
\rho &: Q(L) \rightarrow  \mathrm{Conj}(\mathrm{Env}(Q(L)))\\
      &\hspace{1cm}  q \mapsto \widetilde{q}
\end{align*}
is injective if and only if $L$ is prime.
\end{theorem}

Therefore, by Theorem \ref{prime}, if $L$ is an oriented, prime link in $\mathbb{S}^3$ then the link quandle $Q(L)$ is a subquandle of $\mathrm{Conj}(\mathrm{Env}(Q(L)))\cong \mathrm{Conj}(\pi_1(\mathbb{S}^3\setminus L))$.

In \cite[Proposition 3.3]{BE}, the authors showed that:
\begin{proposition}\label{conj}
The conjugation quandle of any bi-circularly orderable group $G$
 is right orderable.
 \end{proposition}

Hence, this proposition recover a result in \cite[Proposition 7]{DDHPV} and \cite[Proposition 3.4]{BPS2}.
\begin{corollary}\label{bi-order}
    The conjugation quandle of any bi-orderable group $G$
 is right orderable.
\end{corollary}

Therefore, by combining Proposition \ref{conj} and Theorems \ref{isomp} and \ref{prime} we have the following:

\begin{corollary}\label{bicircu}
    If $L$ is an oriented, prime link in $\mathbb{S}^3$ such that the link quandle $Q(L)$ is not right orderable, then $\pi_1(\mathbb{S}^3\setminus L)$ is not bi-circularly orderable. In particular, $\pi_1(\mathbb{S}^3\setminus L)$ is not bi-orderable.
\end{corollary}

In Question 1.1 of \cite{CDN},  the authors asked, can non-bi-orderability of a knot group be determined by examining knot invariants other than the Alexander polynomial. Our Corollary \ref{bicircu} give a positive answer to Question 1.1 in \cite{CDN}.

Let $T(m, n)$ be a torus link with $m, n\geq 2$. By \cite[Proposition 7.1]{RSS}, the link quandle of the torus link $T(m, n)$ has a set of generators $\{a_1, a_2, \cdots, a_m\}$ with relations 
$$a_i=a_{n+i}*a_n*a_{n-1}*\cdots *a_2*a_1,\;\; {\rm for\; any} 1\leq i\leq m$$
where $a_{mj+k}=a_k$ for any $j\in \mathbb{Z}$ and $1\leq k \leq m.$
\begin{theorem}\label{proptorus}\cite{RSS} Let $m, n\geq 2$ such that one is not multiple of the other.
    If $T(m, n)$ is a nontrivial torus link, then its link quandle is not right orderable.
\end{theorem}

Hence, we have the following result:
\begin{corollary}
    The knot group of a nontrivial torus knot is not bi-circularly orderable.
\end{corollary}

\section*{Acknowledgement} 
Mohamed Elhamdadi was partially supported by Simons Foundation collaboration grant 712462.

\end{document}